\documentclass[sigconf]{acmart}

\usepackage{balance}

\def\BibTeX{{\rm B\kern-.05em{\sc i\kern-.025em b}\kern-.08emT\kern-.1667em\lower.7ex\hbox{E}\kern-.125emX}}

 \newtheorem{defi}{Definition}
\newtheorem{thm}{Theorem}
\newtheorem{prop}{Proposition}

\newtheorem{lem}{Lemma}

\copyrightyear{2023}
\acmYear{2023}
\setcopyright{acmcopyright}\acmConference[ISSAC '23]{Proceedings of the 2023 International Symposium on Symbolic and Algebraic Computation}{July 24--27, 2021}{Tromso, Norway}
\acmBooktitle{Proceedings of the 2023 International Symposium on Symbolic and Algebraic Computation (ISSAC '23), July 24--27, 2023, Tromso, Norway}
\acmPrice{15.00}
\acmDOI{10.1145/3452143.3465540}
\acmISBN{978-1-4503-8382-0/21/07}

\begin{document}

\fancyhead{}

\title{Hyperelliptic Integrals to Elliptic Integrals}


\author{Thierry Combot}
\affiliation{%
  \institution{IMB, Universit\'e de Bourgogne}
  \streetaddress{6 avenue Alain Savary}
  \city{Dijon}
  \country{France}}
\email{thierry.combot@u-bourgogne.fr}

%

%
\begin{abstract}
Consider a hyperelliptic integral $I=\int P/(Q\sqrt{S}) dx$, $P,Q,S\in\mathbb{K}[x]$, with $[\mathbb{K}:\mathbb{Q}]<\infty$. When $S$ is of degree $\leq 4$, such integral can be calculated in terms of elementary functions and elliptic integrals of three kinds $\mathcal{F},\mathcal{E},\Pi$. When $S$ is of higher degree, it is typically non elementary, but it is sometimes possible to obtain an expression of $I$ using also elliptic integrals when the Jacobian of $y^2=S(x)$ has elliptic factors. We present an algorithm searching for elliptic factors and a modular criterion for their existence. Then, we present an algorithm for computing an expression of $I$ using elliptic integrals, which always succeed in the completely decomposable Jacobian case.
\end{abstract}

%
%



%

%
\maketitle

\section{Introduction} 

Consider an integral of the form
$$I=\int \frac{P(x)}{Q(x)\sqrt{S(x)}} dx$$
where $P,Q,S\in\mathbb{K}[x]$, $S$ and $\mathbb{K}$ is a finite extension of $\mathbb{Q}$. We can assume that $S$ has no multiple roots, and up to homographic variable change that $S$ is of odd degree, and up to multiplication of $I$, that $S$ is monic.

When $\deg S=1$, the integral $I$ can be expressed respectively with an algebraic function and logs, and thus is elementary. Since works of Legendre \cite{labahn1997reduction}, we know that when $\deg S=3$, this is typically not possible, but to express such integrals it is enough to introduce $3$ new functions, which will note
$$\mathcal{F}(z\mid \kappa)=\! \int\! \frac{dz}{\sqrt{z(z-1)(z-\kappa)}},\;\mathcal{E}(z\mid \kappa)=\! \int\! \frac{z dz}{\sqrt{z(z-1)(z-\kappa)}},$$ 
$$\Pi(z,u\mid \kappa)=\! \int\! \frac{\sqrt{u(u-1)(u-\kappa)}dz}{(z-u)\sqrt{z(z-1)(z-\kappa)}}.$$
These are called respectively elliptic integrals of the first, second and third kind. The $\kappa$ parameter is related to the $j$ invariant of the elliptic curve $y^2=z(z-1)(z-\kappa)$ by the relation $j\kappa^2(\kappa-1)^2=256(\kappa^2-\kappa+1)^3$. Each elliptic integral has a special property, namely
\begin{itemize}
\item The first kind has no poles at all, and thus the integral is a smooth (but multivalued) function over an elliptic curve.
\item The second kind only has a pole at infinity and no residue, thus the integral has a single meromorphic pole at infinity and is a smooth (but multivalued) function elsewhere.
\item The third kind has two simple poles at $z=u$ (which defines two points on the curve $y^2=z(z-1)(z-\kappa)$), with residues $1,-1$. The integral behaves as a log near $z=u$, and is also multivalued along closed loops in the elliptic curve.
\end{itemize}
The first and second kind integral are never elementary, but the third kind can be elementary when $(u,\sqrt{u(u-1)(u-\kappa)})$ is a torsion point of the elliptic curve $y^2=z(z-1)(z-\kappa)$.

Elliptic integrals are special functions, not elementary, but still ``nice'' \cite{lawden2013elliptic}. Among other properties, their continuation in the complex domain is well understood, the inverse of the first kind gives the Weierstrass $\wp$ function, which is doubly periodic and meromorphic. When $\deg S\geq 5$, integrals $I$ are called hyperelliptic, and are much more mysterious and complicated. At the end of XIXth century and beginning of XXth, there was a lot of research \cite{mcdonald1906problem,gilespie1900reduction,roberts1871tract,goursat1885reduction} on possible formulas allowing to reduce such hyperelliptic integral to elliptic ones. To this day, CAS systems as Maple still do not use these works, as for example the following integral returns
$$\int \frac{dx}{\sqrt{1-x^8}}=x {}_2F_1(1/8, 1/2, 9/8, x^8)$$
The answer uses a hypergeometric function, which is typically a more complicated function than the integral given: it is typically a non Liouvillian function, which for special values, admits an integral expression. In fact, the calculation in the opposite way would be more reasonable, which is done by Kovacic algorithm. Still, this integral can in fact be expressed through elliptic integrals of the first kind
$$\frac{1+i-i\sqrt{2}}{4} \mathcal{F}\left( \left. \frac{i(2\sqrt{2}-3)(i\sqrt{2}+1+i+2x)}{(2x+1)(\sqrt{2}i-1-i-2x)} \right\vert 3-2\sqrt{2}\right)+$$
$$\frac{1+i+\sqrt{2}}{4} \mathcal{F}\left( \left. \frac{-\sqrt{2}-1-i-2ix}{(2x+1)(2x+1-i+i\sqrt{2})} \right\vert 3+2\sqrt{2}\right)$$
From the 50s, the interest then shifted to geometrical properties of an algebraic variety, the Jacobian, attached to the hyperelliptic curve $y^2=S(x)$, like Serre question \cite{paulhus2017completely}, and more recently hyperelliptic curves with automorphisms \cite{paulhus2008decomposing,joshi2020hypergeometric,lombardo2021decomposing}.

\begin{defi}
A divisor $D$ on a hyperelliptic curve $\mathcal{C}:y^2=S(x)$ is a function $\mathcal{C} \rightarrow \mathbb{Z}$ with finite support. We say that $D$ is principal if there exists a rational function $f$ on $\mathcal{C}$ such that $\hbox{val}_{z=z_0} f(z)=D(z_0),\;\forall z_0\in\mathcal{C}$.
The set of divisors modulo principal divisors defines an algebraic group of dimension $g=(\deg S-1)/2$, which is called the Jacobian of $\mathcal{C}$.
\end{defi}

The Jacobian $\mathcal{J}$ is an Abelian variety of dimension $g$, where $g$ is the genus of the curve, see \cite{griffiths1989introduction}. In genus $1$, this is the elliptic case, and the Jacobian is the elliptic curve itself. In higher genus, this is an Abelian variety, and as an algebraic group, it can sometimes decompose as a product of (irreducible) Abelian varieties of smaller dimension
$$\mathcal{J}= \mathcal{A}_1\times \dots \times \mathcal{A}_n$$
This factorization is unique up to isogeny of the factors. If a factor is of dimension $1$, then it is an elliptic curve.

\begin{defi}
We say the Jacobian has an elliptic factor $\mathcal{E}$ if $\mathcal{E}$ appears in the decomposition of $\mathcal{J}$. We called rank the number of such factors, and we say that $\mathcal{J}$ is completely decomposable if $\mathcal{J}$ is a product of elliptic curves.
\end{defi}

Elliptic factors can be detected as morphisms from the hyperelliptic curve $\mathcal{C}$ to an elliptic curve $\mathcal{E}$ (see Proposition \ref{prop1}. If the Jacobian is completely decomposable, then $g$ independent morphisms exist, and then all divisors in the Jacobian can be written as combinations of divisors of elliptic integrals composed by such morphisms. The independence condition allows to forbid composition of a morphism by several isogenies between elliptic curves, which would still count as one in the factorisation of the Jacobian). We will first present an algorithm to find such elliptic factors.

\begin{thm}\label{thm1}
Given a curve $y^2=S(x)$ and a bound $m$, algorithm \underline{\textit{EllipticFactors}} computes a maximal list of elements of the form $[\kappa,F,G]$ such that
\begin{itemize}
\item We have the relation $SG^2=F(F-1)(F-\kappa)$ with $F,G\in\overline{\mathbb{K}}(x),\kappa  \in \overline{\mathbb{K}}\setminus \{0,1\}$.
\item The degrees of numerator, denominator of $F,G$ are $\leq m$.
\item The morphisms are independent, i.e. the polynomials $F'/G$ are independent over $\mathbb{C}$.
\end{itemize}
\end{thm}

Remark that the possible morphisms are not unique. Indeed, the elliptic factors in the Jacobian are unique up to isogeny, but there are infinitely many isogenous curves to a given elliptic curve: an elliptic curve is defined up to isomorphism by its j-invariant $j$, but infinitely many isogenous curves can be obtained through roots of the modular polynomials \cite{broker2012modular}. We also need an upper bound on the number of elliptic factors.

\begin{thm}\label{thm2}
Given a curve $y^2=S(x)$ and a prime $p$, the algorithm \underline{\textit{HyperellipticZeta}} computes the maximal number of degree $2$ factors of the Jacobian zeta function
$$\mathcal{Z}_{p^k}(T),\quad k\in\mathbb{N}^*,$$
which gives an upper bound on the rank of the Jacobian.
\end{thm}

The situation is similar to the algorithm mwrank \cite{cremona2001classical} to compute Mordell Weil rank on elliptic curves over $\mathbb{Q}$. A direct search can be done, and modular upper bounds can be found. The question of course is if the rank always coincide to this upper bound for some prime $p$. This is not the case in general \cite{hashimoto1995shimura}, although for $\mathbb{K}=\mathbb{Q}$ no example is known. In all examples tested, it was always possible to find an upper bound matching the number of elliptic morphisms found. We can now present the effective calculation of hyperelliptic integrals.

\begin{thm}\label{thm3}
If $y^2=S(x)$ has a completely decomposable Jacobian, then any integral $I=\int P/(Q\sqrt{S})$ can be expressed using algebraic functions, logarithms and elliptic integrals. The algorithm \underline{\textit{HyperellipticToElliptic}} computes such an expression.
\end{thm}

If the Jacobian is not completely decomposable, not all integrals $I$ with radical $\sqrt{S}$ can be expressed with elliptic integrals, and the algorithm can fail. In such case, we do not guarantee that $I$ does not admit an elliptic expression, as a more careful analysis of $\mathcal{J}/(\mathcal{E}_1\times\dots\times \mathcal{E}_n)$ would be necessary. This is a Prym variety, and similar to what Trager \cite{trager2022integration} did, we should test for torsion of divisors in this Abelian variety, and if yes then decompose divisors in a torsion part and in a part in $\mathcal{E}_1\times\dots\times \mathcal{E}_n$.

The plan of the article is the following. In section $2$, we will present the algorithm to find elliptic factors and the modular bounding algorithm. In section $3$, we will present the explicit construction of the morphism $\mathcal{E}_1\times\dots \mathcal{E}_n \rightarrow \mathcal{J}$, and its application to generalized Hermite reduction and integration in the completely decomposable case.

\section{Computing Elliptic Factors}

\subsection{Jacobians and elliptic morphisms}

At the moment, we have two notions. The elliptic morphisms, which are rational functions from the hyperelliptic curve to an elliptic curve, and the elliptic factors in the Jacobian. Let us prove that they correspond to each other.

\begin{prop}\label{prop1}
The Jacobian of a hyperelliptic curve has $r$ elliptic factors $\mathcal{E}_1,\dots,\mathcal{E}_r$ if and only if there exist $n$ non constant morphisms $(F_i(x),yG_i(x))$ from the hyperelliptic curve to $\mathcal{E}_i$ such that the polynomials $F_i'/G_i$ are independent over $\mathbb{C}$.
\end{prop}

Remark, in the following, we will compose abusively an elliptic integral by $F_i(x)$, although rigorously, we would need to precise the other component $yG_i(x)$. This abuse is traditional as it is used for writing composition of integrals in CAS systems. However, this leaves an ambiguity on the sign of $G_i(x)$, and non careful simplifications results in random non consistent branch choices for the square roots. In the implementation, to remember this sign choice, a parameter $\epsilon_i=\pm 1$ is added in front of the integrals.

\begin{proof}
Let us first begin by assuming that we have $r$ non constant morphisms $(F_i(x),yG_i(x))$ with $F_i'/G_i$ are independent over $\mathbb{C}$. Let us first check that $F_i'/G_i$ are indeed polynomials.

Considering the elliptic integral $\mathcal{F}(F_i(x) \mid \kappa)$, this is a smooth function everywhere on the elliptic curve (although multivalued), and thus its derivative in $x$ is smooth everywhere outside points with vertical tangents. Its derivative is given by
$$\mathcal{F}(F_i(x) \mid \kappa)'= F_i'(x)/(G_i(x)\sqrt{S(x)}).$$
Thus the rational fraction $F_i'(x)/G_i(x)$ cannot have poles outside roots of $S$. At roots of $S$, if $F_i'(x)/G_i(x)$ had a pole, then the singularity would be of order at least $3/2$, and thus after integration, $\mathcal{F}(F_i(x) \mid \kappa)$ would have a singular point. Thus $F_i'(x)/G_i(x)$ has no poles at all, and so is a polynomial. Moreover, as it has no pole at infinity, its degree is $\leq g-1$.

We consider the Abel Jacobi map
$$\sum\limits_{i=1}^g \int^{x_i} \frac{x^{j-1}}{\sqrt{S(x)}} dx = t_j,\;\;j=1\dots g$$
The Abel Jacobi theorem \cite{griffiths1989introduction} states that this map going from the Jacobian to $\mathbb{C}^g/\Lambda$ can be inverted and gives a meromorphic parametrization of the Jacobian. The divisors support points abscissa are given by the $x_i$, and thus the symmetric functions in the $x_i$ can be written meromorphically in the $t_i$.

So now making a linear combination of the lines, we can make appear the polynomials $F_i'(x)/G_i(x)$ in the numerators in the first lines. Replacing also the $t_i$ by a suitable linear combination of the $t_i$, the system becomes
$$\sum\limits_{i=1}^g \int^{x_i} \frac{F_j'(x)}{G_j(x)\sqrt{S(x)}} dx = t_j,\; j=1\dots r$$
$$\sum\limits_{i=1}^g \int^{x_i} \frac{x^{j-1}}{\sqrt{S(x)}} dx = t_j,\; j=r+1\dots g$$
Noting $\oplus$ the addition law on elliptic curves the $n$ first lines become
$$\int^{F_j(x_1)\oplus \dots \oplus F_j(x_g)} \frac{1}{\sqrt{z(z-1)(z-\kappa_j)}} dz = t_j$$
and thus can be inverted using Weierstrass $\wp$ function
$$F_i(x_1)\oplus \dots \oplus F_i(x_g) = \wp_{\kappa}(t_i)$$
Thus the symmetric functions in the $x_i$ can be obtained as algebraic expressions in $\wp_{\kappa}(t_i)$ and an Abelian function coming from the inversion of the last $g-n$ equations. As this inverse is in fact meromorphic, moving the $t_1,\dots,t_r$ with fixed $t_{r+1},\dots t_g$ defines an Abelian subvariety of dimension $r$ of the Jacobian, and the $t_1,\dots,t_n$ parametrize it such that addition law on the Jacobian corresponds to the classical addition of the parameters $t_j$. Thus the Jacobian factors by $\mathcal{E}_1\times,\dots,\times \mathcal{E}_r$.

Conversely, assume that the Jacobian factors by $\mathcal{E}_1\times,\dots,\times \mathcal{E}_r$. Then it can be parametrized by $r$ Weierstrass functions and a $g-r$ variables Abelian function $A$. Substituting, we have relations for $j=1\dots g$
$$\sum\limits_{i=1}^g \int^{\Phi_i(\wp(t_1),\dots,\wp(t_r),A(t_{r+1},\dots,t_g))} \frac{x^{j-1}}{\sqrt{S(x)}} dx = \sum\limits_{k=1}^g a_{ik} t_k$$
where $\Phi_i$ are algebraic functions. The coefficients $a_{ik}$ determine the decomposition of the Jacobian seen as the quotient $\mathbb{C}^g/\Lambda$, where elliptic factors are simply straight lines. Noting
$$s_i=\wp(t_i),\; i=1\dots r,\quad (s_{r+1},\dots,s_g)=A(t_{r+1},\dots,t_g),$$
and making linear combinations on the lines to obtain on the right side $t_1,\dots,t_r$, we have for $j=1\dots r$
$$\sum\limits_{i=1}^g \int^{\Phi_i(s_1,\dots,s_g)} \frac{P_j(x)}{\sqrt{S(x)}} dx = \int^{s_j} \frac{1}{\sqrt{z(z-1)(z-\kappa_j)}} dz$$
where $P_j(x)$ are polynomials of degree $\leq g-1$. We can now invert algebraically the $\Phi_i$, giving for $j=1\dots r$
$$\sum\limits_{i=1}^g \int^{x_i} \frac{P_j(x)}{\sqrt{S(x)}} dx = \int^{\Phi_j^{-1}(x_1,\dots,x_g)} \frac{1}{\sqrt{z(z-1)(z-\kappa_j)}} dz.$$
Fixing $x_2,\dots,x_g$ to some generic constants, this gives $r$ algebraic morphisms from the hyperelliptic curve to the elliptic curves $y^2=z(z-1)(z-\kappa_j)$.

Let us proof that, up to a small transformation, algebraic morphisms are rational. If
$$\int \frac{P(x)}{\sqrt{S(x)}} dx= \int^{f(x)} \frac{1}{\sqrt{z(z-1)(z-\kappa)}} dz$$
then the monodromy constants of the hyperelliptic integral belong to a lattice whose points are such that a multiple of them belong to $\Lambda'$, the lattice associated to the elliptic integral on the right. The possible ratios are bounded by the algebraic degree of $f$. Thus multiplying by a suitable integer $k\in\mathbb{N}^*$ and applying $\wp$, we have
$$\wp\left( k \int \frac{P(x)}{\sqrt{S(x)}} dx \right)=k.f(x)$$
where $.$ denotes the multiplication on elliptic curves. By construction, the left hand side is meromorphic, the right hand side is algebraic. Thus $k.f$ is rational.

Let us now consider a rational morphism $(F,G)$ from $y^2=S(x)$ to an elliptic curve. It is rational in $x,y$, and can be assumed to be polynomial in $y$ of degree $1$. Building the other morphism $(F,G)(x,y) \oplus (F,G)(x,-y)$, this function is by construction symmetric in $y$, and thus is function of $x$ only. If this function was non constant, this would produce a rational parametrization of $v^2=u(u-1)(u-\kappa)$, which is impossible as it is an elliptic curve. Thus $(F,G)(x,y) \oplus (F,G)(x,-y)$ is constant, and up to a translation on the elliptic curve, we can assume it is the point $O$ (the point at infinity). Thus any rational elliptic morphism can be written
$$(F(x),yG(x)),\quad F,G\in\mathbb{C}(x)$$
Finally, we have obtained $r$ rational morphisms such that
$$\int \frac{k_jP_j(x)}{\sqrt{S(x)}} dx= \int^{F_j(x)} \frac{1}{\sqrt{z(z-1)(z-\kappa_i)}} dz$$
and by construction, the polynomials $P_j$ and so the polynomials $k_jP_j$ are independent over $\mathbb{C}$. 
\end{proof}

\subsection{Finding Elliptic Morphisms}

The condition for $(F_i(x),yG_i(x))$ being an elliptic morphism is that it satisfies
$$SG^2=F(F-1)(F-\kappa),\;\; \kappa\neq 0,1,\; F\notin\mathbb{C}.$$
Such an equation can be solved up to degree $m$ by brute force using Groebner basis, requiring typically $4m$ unknowns with equations of degree $4$. But there is a better approach. Noting $F=U/V$, we see that the simple roots of $U,V,U-V,U-\kappa V$ are roots of $S$, as $U,V,U-V,U-\kappa V$ cannot have common roots. The other roots should have even multiplicities, and thus this equation can be transformed into
\begin{equation}\begin{split}\label{eqUV}
U=A^2\prod_{i\in\mathcal{R}_1} (x-\alpha_i),\; V=B^2\prod_{i\in\mathcal{R}_2} (x-\alpha_i)\\
U-V=C^2\prod_{i\in\mathcal{R}_3} (x-\alpha_i),\; U-\kappa V=D^2\prod_{i\in\mathcal{R}_4} (x-\alpha_i)
\end{split}\end{equation}
where $\mathcal{R}$ is a partition of the root indices of $S$. Now, for each choice of $\mathcal{R}$, this is a system with at most $2m$ unknowns with quadratic equations.

\begin{prop}\label{propfinite}
For a given degree bound $m$, there are finitely many non constant solutions to $SG^2=F(F-1)(F-\kappa)$ with $\kappa\neq 0,1$.
\end{prop}

\begin{proof}
For a given $\kappa\neq 0,1$, the equation $SG^2=F(F-1)(F-\kappa)$ is an elliptic Diophantine equation over $\mathbb{C}(x)$. The Mordell Weil theorem applies, and thus a finite basis of solutions exist, and thus countably many. If infinitely many existed of the same degree, then the Zariski closure of this set would define a one parameter family of solutions, and thus uncountably many. Thus for each degree, only finitely many exist.

If a solution for a given $\kappa$ exists, then the corresponding elliptic curve $y^2=z(z-1)(z-\kappa)$ should factor in the Jacobian. There are finitely many elliptic factors up to isogeny. However, there are infinitely many elliptic curves isogenous to a given elliptic curve. The relation between their j-invariant is given by modular polynomials, which are countably many. As $j\kappa^2(\kappa-1)^2=256(\kappa^2-\kappa+1)^3$, there are countably many possible $\kappa$. As before, if there were infinitely many realisable $\kappa$ for a given degree, taking Zariski closure would give a continuum of solutions, and thus uncountably many $\kappa$. Thus only finitely many $\kappa$ are possible.
\end{proof}

To use effectively Proposition \ref{propfinite}, we need to ensure that a rational function $U/V$ is uniquely represented. Adding to \eqref{eqUV} the additional constraints
$$\kappa(\kappa-1)\hbox{resultant}(A,B)\neq 0,\; A^2=\mu \tilde{A}^2,\;\tilde{A},B \hbox{ unitary} $$
will give an algebraic system with finitely many solutions, and thus be defined by a zero dimensional ideal. Let us present the algorithm.\\

\noindent
\underline{\textit{EllipticFactors}}\\
\textsf{Input:} A hyperelliptic curve $y^2=S(x)$, a bound $m\geq 2$.\\
\textsf{Output:} A list of $[\kappa,F,G]$ of independent solutions of $SG^2=F(F-1)(F-\kappa)$, $F,G\in\overline{\mathbb{K}}(x),\kappa\neq 0,1$
\begin{enumerate}
\item Compute all partitions $\mathcal{R}$ of $[1,\dots,\deg S]$ in $4$ sets.
\item Compute the splitting field $\mathbb{L}$ contemning the roots of $S$.
\item $L:=\{ \}$. For each partition in $\mathcal{R}$ do
\begin{enumerate}
\item Note $Id_{d_1,d_2}=($
$$\hbox{coeffs}_x(\mu A^2\prod_{i\in\mathcal{R}_1} (x-\alpha_i)-       B^2\prod_{i\in\mathcal{R}_2} (x-\alpha_i)-C^2\prod_{i\in\mathcal{R}_3} (x-\alpha_i)),$$
$$\hbox{coeffs}_x(\mu A^2\prod_{i\in\mathcal{R}_1} (x-\alpha_i)-\kappa B^2\prod_{i\in\mathcal{R}_2} (x-\alpha_i)-D^2\prod_{i\in\mathcal{R}_4} (x-\alpha_i)),$$
$$\kappa(\kappa-1)\hbox{resultant}(A,B)\epsilon+1)$$
where $A,B$ are generic unitary polynomials of degree $d_1,d_2$, $C,D$ generic polynomials of degrees $\max(d_1+\sharp \mathcal{R}_1/2,$ $d_2+\sharp \mathcal{R}_2/2)-\sharp \mathcal{R}_3/2,\; \max(d_1+\sharp \mathcal{R}_1/2,d_2+\sharp \mathcal{R}_2/2)-\sharp \mathcal{R}_4/2$ respectively.
\item For $2d_1+\sharp \mathcal{R}_1,2d_2+\sharp \mathcal{R}_2\leq m$, compute the prime decomposition of $Id_{d_1,d_2}$. Extend the field $\mathbb{L}$ to contain all its solutions
\item For each solution, check if $F'/G$ is linearly independent with the $L_{i,2}'/L_{i,3}$, and if so, add $[\kappa,F,G]$ to $L$.
\end{enumerate}
\item Return $L$
\end{enumerate}

\begin{proof}[Proof of Theorem \ref{thm1}]
We know by construction that any morphism $F=U/V$ will satisfy equations \eqref{eqUV} for some partition choice $\mathcal{R}$. In step $1$, all possible partition are computed. In step $2$, the coefficient field $\mathbb{L}$ is extended to contain all roots of $S$, allowing to define the ideals $Id_{d_1,d_2}$ of step $3$. The degree of $U,V$ should be $\leq m$, and this condition is encoded in step $3b$ by the constraint $2d_1+\sharp \mathcal{R}_1,2d_2+\sharp \mathcal{R}_2\leq m$. By Proposition \ref{propfinite}, we know that this ideal is zero dimensional, thus we can compute one by one its solutions, possibly extending the field $\mathbb{L}$. Step $3c$ check that the found solution is independent from the previous found ones, and if so, adds it to $L$. Thus the elements of $L$ are indeed rational morphisms (as they satisfy equations \eqref{eqUV}), and are independent. No solutions can be missed as all solutions of $Id_{d_1,d_2}$ for all possible $d_1,d_2$ are computed.
\end{proof}

The big drawback of this approach is that the splitting field of $S$ needs to be calculated beforehand. However, \cite{zarhin2000hyperelliptic} proves that Jacobian endomorphisms are trivial when the Galois group is big, which implies that there cannot be many small degree morphisms to isogenus elliptic curves. And thus not many possible j-invariants, and so not many possible $\kappa$. As the number of independent morphisms for the same $\kappa$ is bounded by the genus, if a large field extension was required to write them, then all their conjugates would also give solutions.
This strongly suggests that $S$ with large splitting field will not have any morphisms. The family with a non trivial morphism with the largest Galois group found is the following
$$y^2=P(x^n),\;\; \deg P=3,\; \hbox{Gal}(P)=S_3,\; n\hbox{ odd}$$
in which the Galois group of $P(x^n)$ is generically of size $3(6n)^3$, much smaller than the $(3n)!$ expected.

In practice, we will first apply modular tests before looking for elliptic morphisms, and thus (probably) avoid the worst cases. Also, first predicting the rank allows to cut short the algorithm: once the predicted number is found, we know that there are no more to find and thus we can directly return $L$.

\subsection{Jacobian Zeta Functions}

If there exists an elliptic factor in the Jacobian, we can look at what happens to it when performing a modular reduction. If the morphism can be reduced in $\mathbb{F}_{p^k}$ and the prime $p$ chosen does not divide a denominator nor the discriminant of the curve, then the Jacobian in $\mathbb{F}_{p^k}$ is well defined and will also have an elliptic factor. We now define the Jacobian Zeta function \cite{vercauteren2002computing} by
$$\mathcal{Z}_{p^k}(T)=\exp\left( \sum\limits_{i=1}^\infty \frac{N_i T^i}{i} \right)$$
where $N_i$ is the number of points of $y^2=S(x)$ in $\mathbb{F}_{p^{ki}}$. The Jacobian Zeta function is rational, and can always be written
$$\mathcal{Z}_{p^k}(T)=\frac{\Psi_{p^k}(T)}{(1-p^kT)(1-T)},\;\Psi_{p^k}\in \mathbb{Z}[T]$$
Moreover, $\Psi$ is of degree $2g$, and its roots have modulus $p^{k/2}$. The main property is that if the Jacobian factorize mod $p$, then the $\Psi_{p^k}$ function also factorizes, and its factors are the $\Psi$ functions of the Abelian factors of the Jacobian. As for an elliptic curve the $\Psi$ function will have degree $2$, we can bound the number of factors by the number of degree $2$ factors of $\Psi_{p^k}$. Moreover, the Zeta function of an elliptic curve uniquely determines an elliptic curve up to an isogeny (mod $p$) as two elliptic isogenous curves have the same number of points, thus two isogenous factors will give the same factor of $\Psi_{p^k}$. Some efficient algorithms \cite{kedlaya2004computing} exist to compute this function, but here in our actual implementation in Maple we use a simple counting points method which is sufficient.

To first check the bad $p$'s, if a $p$ divides a denominator of an elliptic morphism, it degenerates, and thus the dimension of the Jacobian reduces. This implies that the genus reduces, and so that a double root in $S$ has appeared. This can be checked using the discriminant. Bounding $k$ is much more delicate, as algebraic extensions do appear in the morphisms. In fact, in algorithm \underline{\textit{EllipticFactors}}, not only the splitting field of $S$ is needed, but also subsequent field extensions can be necessary for the solutions of ideals $\hbox{Id}_{d_1,d_2}$. These, with a growing $m$, can be arbitrary large. It is thus necessary to study the factorization of the $\mathcal{Z}_{p^k}(T)$ for all $k\in\mathbb{N}^*$.\\

\noindent
\underline{\textit{HyperellipticZeta}}\\
\textsf{Input:} A hyperelliptic curve $y^2=S(x)$, a prime $p$.\\
\textsf{Output:} An integer which is an upper bound on the rank of the Jacobian.\\
\begin{enumerate}
\item If $p$ divides the discriminant of $S$, return $\infty$.
\item Compute the Jacobian Zeta function $\mathcal{Z}_p(T)$ and then $\Psi_p(T)$ for the curve $y^2=S(x)$
\item Compute the polynomials
$$P_k=\prod_{i=1}^k \Psi_p(\xi^i T^{1/i})$$
for $k\in\mathbb{N}^*,\; \phi(k) \leq \deg S -1$.
\item For each possible $i$, factorize $P_i$ in $\mathbb{Q}[T]$, compute the number of factors of degree $\leq 2$, counting those of degree $1$ as $1/2$.
\item Return the maximum of these numbers. 
\end{enumerate}

\begin{proof}[Proof of Theorem \ref{thm2}]
Let us first remark that
$$\mathcal{Z}_{p^k}(T^k)=\exp\left( \sum\limits_{i=1}^\infty \frac{N_{ik} T^{ik}}{i} \right)=\prod\limits_{j=1}^k \exp\left( \sum\limits_{i=1}^\infty \frac{N_i (\xi^j T)^i}{ik} \right)$$
with $\xi$ a $k$ th root of unity. This gives the relation
$$\frac{\Psi_{p_k}(T^k)}{(1-T^k)(1-p^kT^k)}=\prod_{j=1}^k \frac{\Psi_p(\xi^j T)}{(1-\xi^j T)(1-p\xi^j T)}$$
Thus $\Psi_{p_k}(T^k)=\prod_{j=1}^k \Psi_p(\xi^j T)$ , and so step $3$ computes $P_k=\Psi_{p^k}(T)$.

Let us now consider an irreducible factor of $\Psi_p(T)$. The Galois group acts transitively on its roots $\alpha_1,\dots,\alpha_m$. The transformation
$$P(T) \rightarrow \prod_{j=1}^k P(\xi^j T^{1/k})$$
is multiplicative, and thus this factor will also appear as a factor of $\Psi_{p^k}(T)$. The roots of this factor will be $\alpha_1^k,\dots,\alpha_m^k$. As long as they are distinguishable, the Galois group acts transitively on these roots, and thus the minimal polynomial vanishing on them is still irreducible. But they can become undistinguishable if some of them are equal. This happens if $\alpha_i=\xi \alpha_j$ for some $\xi$ $k$th root of unity. We consider the minimal $k$ for which this happens, and thus $\xi$ will be a primitive $k$-root of unity. So, if we have such relation, we will also obtain $\phi(k)$ other conjugate relations, and thus we need at least $\phi(k)$ roots. As $\Psi_p(T)$ has degree $2g=\deg S-1$, we know that such relation can only occur for the first time when $\phi(k) \leq \deg S -1$ (this can occur for all the multiples of these $k$, but such relation between roots would have already been detected).

Thus new factorizations can only occur for $\Psi_{p^k}(T)$ with $\phi(k) \leq \deg S -1$, and so it is enough to consider only them. Now knowing that all roots have modulus $p^{k/2}$, this means that we can write
$$\Psi_{p^k}(T)=\prod_{j=1}^{g} (T^2+a_jT+p^k),\quad a_j\in\overline{\mathbb{Q}}$$
If $\Psi_{p^k}(T)$ has a factor of degree $1$, so a root in $\mathbb{Q}$, it vanishes one of the quadratic factors, for which then $a_j\in\mathbb{Q}$, and so automatically $\Psi_{p^k}(T)$ has another factor of degree $1$. So a factor of degree $1$ is always a a factor of a factor of degree $2$ of the form $T^2+a_jT+p^k$. So counting degree $1$ factors for $1/2$ still ensures that the answer of step $4$ will be an integer, and correctly counts the number of rational $a_j$. The step $5$ then consider the maximum of these numbers, as we do not know a priori which $k$ we should choose.
\end{proof}

\noindent
\textbf{Example}: We consider the hyperelliptic curves $y^2=x^n-1$. For $n$ odd, this already satisfies our assumption. For $n$ even, we can make the transformation $x \leftarrow x^{-1}+1$, giving an isomorphic curve $y^2=x^n((x^{-1}+1)^n-1)$ of odd degree. Below we compute the rank, morphisms degree, a certifying prime, and the timings to find the elliptic factors and certify the rank.
\begin{center}

\begin{tabular}{|c|c|c|c|c|}\hline
$n$& Rank & degree & prime & Time   \\\hline
5  & 0    &        &  11   & $0.2s$ \\\hline
6  & 2    &  $2,2$ &       & $0.7s$ \\\hline
7  & 0    &        &  29   & $63s$  \\\hline
8  & 3    & $2,2,2$&       & $16.3s$\\\hline
9  & 1    & $3$    &  19   & $425s$ \\\hline
10 & 0    &        &  11   & $45s$  \\\hline
11 & 0    &        &  23   & $>1000s$  \\\hline
12 & 5    &$3,4,4,4,4$&    & $557s$ \\\hline
\end{tabular}\\
\end{center}

Using these morphisms, we can for example compute the genus $5$ hyperelliptic integral
$$\int \frac{(x^3+x^2+x+1)dx}{\sqrt{1-x^{12}}}=\frac{1+i}{6} \mathcal{F}\left(\left. \frac{-i(x^3-1)}{x^3+1}\right| -1\right)$$
$$+\frac{3^{1/4}\sqrt{2}((1-i)+(1+i)\sqrt{3})}{24}\mathcal{F}\left(\left. \frac{-i\sqrt{3}x^2}{x^4-1}\right| -1\right)$$
$$-\frac{\sqrt{\sqrt{3}i+3}\sqrt{2}(-1+\sqrt{3}i)}{24} \mathcal{F}\left(\left. \frac{-4ix^2\sqrt{3}}{(x^2-1)^2\sqrt{3}i+x^4-1}\right| 2\right)$$
$$-\frac{\sqrt{\sqrt{3}i+3}\sqrt{2}(-3+\sqrt{3}i)}{48} \mathcal{F}\left(\left. \frac{-3+i\sqrt{3}i}{2x^4-2}\right| \frac{1}{2}-i\frac{\sqrt{3}}{2}\right)$$
$$-\frac{\sqrt{\sqrt{3}i-3}\sqrt{2}(3+\sqrt{3}i)}{48} \mathcal{F}\left(\left. \frac{(\sqrt{3}i+3)x^4}{2x^4-2}\right| \frac{1}{2}+i\frac{\sqrt{3}}{2}\right)$$

\section{Decomposition of Hyperelliptic Integrals}

\subsection{Hermite Reduction}

The decomposition happens in three parts. First we will remove the multiple poles of $\int P/(Q\sqrt{S}) dx$. The novelty is that we also have to our disposal elliptic integral of the second kind, which after composition, will provide new meromorphic function in addition to the classical algebraic functions.\\

\noindent
\underline{\textit{HermiteReduction}}\\
\textsf{Input:} A hyperelliptic integral $I=\int P/(Q\sqrt{S}) dx$, and a complete list $L$ of independent elliptic morphisms..\\
\textsf{Output:} An integrand $J$ with simple poles and a function $H$ such that $I'-H'=J$, or FAIL.\\
\begin{enumerate}
\item Note $[\kappa_i,F_i,G_i]=L_i,\; i=1\dots \sharp L$
\item Note $\tilde{Q}$ the lcm of $Q$ and the denominators of $F_i$, and $\hat{Q}$ the unitary polynomial whose factors are those of $\tilde{Q}$ with multiplicity one less.
\item Compute the partial fraction decomposition of
$$\tilde{J}=\sum_{i=1}^{\sharp L} a_i \frac{F_iF_i'}{G_i} +\sqrt{S(x)} \left( \frac{\sum_{i=1}^\ell b_i x^i}{\hat{Q}(x)\sqrt{S(x)}} \right)' -\frac{P}{Q}  $$
with $\ell=\deg \hat{Q} +\max (0,\deg \hbox{num}(F_i)-\deg \hbox{den}(F_i))$.
\item Compose a system $Sys$ of equations consisting of the $x$ coefficients of the partial fraction decomposition whose poles are either of order $\geq 2$ or all poles which are roots of $S$ (including the pole at infinity.
\item Solve $Sys$. If no solution, return FAIL, else substitute and return
$$H=\sum_{i=1}^{\sharp L}  \int^{F_i(x)}\!\!\!\!\!\!\!\!\frac{a_i zdz}{\sqrt{z(z-1)(z-\kappa_i)}} + \frac{\sum_{i=1}^\ell b_i x^i}{Q(x)\sqrt{S(x)}},\;\; J=\frac{\tilde{J}}{\tilde{Q}\sqrt{S(x)}}$$
\end{enumerate}

\begin{prop}\label{propherm}
Algorithm \underline{\textit{HermiteReduction}} always terminate and returns a correct answer. If $\sharp L$ equals the rank of the Jacobian and the genus, the algorithm never fails.
\end{prop}

\begin{lem}\label{lem2}
If $(\kappa_i,F_i,G_i)_{i=1\dots r}$ define $r$ independent morphisms, then the elliptic integrals of the second kind $\mathcal{E}(F_i(x) \mid \kappa_i)$ are linearly independent over $\mathbb{C}$ modulo addition of a rational function over the hyperelliptic curve.
\end{lem}

\begin{proof}[Proof of Lemma \ref{lem2}]
Consider a basis of cycle on the hyperelliptic curve, representing any cycle by a vector of integers in $\mathbb{Z}^{2g}$. Now the morphisms $(F_i,G_i)$ send such cycle to cycles on elliptic curves $y^2=x(x-1)(x-\kappa_i)$. Again, considering basis of cycles on these elliptic curves, we can represent them by vectors in $\mathbb{Z}^2$. This defines $L_i:\mathbb{Z}^{2g} \rightarrow \mathbb{Z}^2$ a linear application. We now look at the Abel Jacobi map
$$\sum\limits_{i=1}^g \int^{x_i} \frac{F_j'(x)}{G_j(x)\sqrt{S(x)}} dx = t_j,\; j=1\dots r$$
$$\sum\limits_{i=1}^g \int^{x_i} \frac{x^{j-1}}{\sqrt{S(x)}} dx = t_j,\; j=r+1\dots g$$
Considering cycles $\gamma_1,\dots,\gamma_g$ and their representation $X_1,\dots,X_g \in\mathbb{Z}^{2g}$, acting the multivaluation of the first $r$ integrals of the left hand side along these cycles gives
$$t_j \rightarrow t_j + \omega_j L_j\left(\sum_{i=1}^g X_i \right),\;\; j=1\dots r$$
where $\omega_j$ are the period lattice of the elliptic integrals. As the Jacobian is $g$ dimensional, the periods form $2g$ dimensional lattice. If the linear application $\mathcal{L}:(L_1,\dots,L_r): \mathbb{Z}^{2g} \rightarrow \mathbb{Z}^{2r}$ has rank $r'<2r$, then the period lattice is of dimension at most $r'+2g-2r$ (as there can be at most $2$ periods for each $\mathbb{C}$ dimension of the complementary space). This implies that $r'=2r$ and thus that $\mathcal{L}$ is surjective.

Now assume there exists a linear combination $\sum_{i=1}^r a_i\mathcal{E}(F_i(x) \mid \kappa_i)$ equalling to a rational function. As $\mathcal{L}$ is surjective, for any $i_0$, we can find a cycle $\gamma$ such that $\mathcal{L}(\gamma)= (0,\dots,0,1,0,\dots,0)$ with $1$ in the $2i_0$ th position. Computing the linear combination of integrals along this cycle gives zero to all terms except the $i$ th one. As it should be equal to a rational function (thus single valued), this implies that $a_{i_0}=0$. Thus all the $a_i=0$, and so the integrals are independent.
\end{proof}

\begin{proof}[Proof of Proposition \ref{propherm}]
The algorithm always terminates as it executes a fixed number of steps. If the rank equals the genus, then the $g$ independent morphisms $(\kappa_i,F_i,G_i)_{i=1\dots g}$ will generate $g$ independent elliptic integrals of the second kind $\mathcal{E}(F_i(x) \mid \kappa_i)$ according to Lemma \ref{lem2}. Over a hyperelliptic curve of genus $g$, there are $g$ independent hyperelliptic integrals of the second kind, which matches exactly those we have as pull backs of elliptic integrals of the second kind. Thus all of them can be written as a linear combination of the $\mathcal{E}(F_i(x) \mid \kappa_i)$. Thus there exists a linear combination of the form
$$H(x)= R(x)+\sum\limits_{i=1}^{\sharp L} a_i \mathcal{E}(F_i(x) \mid \kappa_i),\quad R\in \mathbb{C}(x,\sqrt{S(x)})$$
such that $H'-P/(Q\sqrt{S})$ has only simple poles.

Let us look at which multiple poles we can have in $R$. There are the poles of $Q$, and the poles of the $F_i$, which appear in the derivative of $\mathcal{E}(F_i(x) \mid \kappa_i)$. The the algebraic part can have poles of order at most one less, and thus its denominator divides the polynomial $\hat{Q}$ computed in step $2$. In step $3$, we differentiate a general linear combination of the $\mathcal{E}(F_i(x) \mid \kappa_i)$ and $(\sum_{i=1}^\ell b_i x^i)/(\hat{Q}(x)\sqrt{S(x)})$. Looking at the pole at infinity, it is enough to take for $\ell=\deg \hat{Q} +\max (0,\deg \hbox{num}(F_i)-\deg \hbox{den}(F_i))$.
Differentiating this and subtracting the integrand should thus have simple poles. This constraint is written as a linear system in step $4$, and solved in step $5$. If a solution is found, then the corresponding integral is returned.
\end{proof}

If the Jacobian is not totally decomposable, then the integrals $\mathcal{E}(F_i(x) \mid \kappa_i)$ may not be sufficient to remove all the multiple poles, but it could still work for some integrals $I$, allowing to perform the next reduction steps.

\subsection{The Elliptic Divisor Variety}

In this subsection, we will compute explicitly a parametrization of the divisors with at most $g$ points on the hyperelliptic curve by $g$ points $(c_i,d_i)$ on elliptic curves.\\

\noindent
\underline{\textit{EllipticDivisors}}\\
\textsf{Input:} A hyperelliptic curve $y^2=S(x)$, a complete list $L$ of independent elliptic morphisms.\\
\textsf{Output:} A polynomial $R$ and an integrand $P/(Q\sqrt{S})$ with at most $g$ simple poles abscissa, smooth at infinity, and with residues $\pm 1$.\\
\begin{enumerate}
\item Note $[\kappa_i,F_i,G_i]=L_i,\; i=1\dots \sharp L$, and $D=[\hbox{num}(F_i-c_i),G_i]_{i=1\dots \sharp L}$
\item Note
$$R=\sum\limits_{i=0}^{\lceil m/2\rceil} a_i x^i +y\sum\limits_{i=0}^{\lfloor (m-\deg S)/2 \rfloor} b_i x^i$$
with $m=g+\sum\limits_{i=1}^{\sharp L} \max(\deg \hbox{num}(F_i),\deg \hbox{den}(F_i))$.
\item Consider the linear system $S$ given by the coefficients in $x$ of
$$\hbox{num}(R(x,d_i/G_i(x)))=0 \hbox{ mod } \hbox{num}(F_i-c_i),\; i=1\dots \sharp L$$
\item Solve $S$ in $a,b$, and substitute a solution in $R$
\item Return $R$ and the simplified $x$ derivative of
\begin{equation}\label{eqpi2}
\sum\limits_{i=1}^{\sharp L} \Pi(F_i(x),c_i \mid \kappa_i)-\left(\ln R\left(x,\sqrt{S(x)}\right)-\ln R\left(x,-\sqrt{S(x)}\right)\right)
\end{equation}
\end{enumerate}

\begin{prop}
Algorithm \underline{\textit{EllipticDivisors}} is correct.
\end{prop}

\begin{proof}
The derivatives of the elliptic integrals of the third kind $\Pi(F_i(x),c_i \mid \kappa_i)$ will have poles at the solutions of $F_i(x)=c_i$. For a generic $c_i$, there will be $\max(\deg \hbox{num}(F_i),\deg \hbox{den}(F_i))$ distinct simple roots. Thus the sum
\begin{equation}\label{eqPi}
\sum\limits_{i=1}^{\sharp L} \Pi(F_i(x),c_i \mid \kappa_i)
\end{equation}
will have for generic $c_i$'s $\sum\limits_{i=1}^{\sharp L} \max(\deg \hbox{num}(F_i),\deg \hbox{den}(F_i))$ distinct poles abscissa (all having residues $\pm 1$). Considering only the ones with residue $+1$ (which depend on the valuation choice for the square root), this defines a divisor with this number of points. Now we know that any divisor can be reduced up to a principal divisor to a divisor with $g$ points. Thus adding to this divisor $g$ (suitably chosen) points will give a principal divisor. In step $2$, we consider $R$ a function on the hyperelliptic curve with this number of roots, $m$.

In step $3$, we build the equations requiring that $R$ vanishes on the solutions of the $F_i(x)=c_i$, in step $4$ we solve this system (which then always has a solution), and substitute in $R$. Now $R(x,\sqrt{S(x)})$ has the same poles as the sum \eqref{eqPi} with residues $+1$, and the conjugate has thus the same poles as the sum \eqref{eqPi} with residues $-1$. Thus the derivative of \eqref{eqpi2} will not have any poles as these points. Stays the $g$ additional roots of $R$, which give $g$ poles in the derivative of \eqref{eqpi2}, with residues $\pm 1$. The result is smooth at infinity as infinity is a branch point, and thus cannot have a residue. As all singularities in \eqref{eqpi2} are logarithmic, they cannot be at infinity, and the poles of the derivative are thus by construction simple.
\end{proof}

If the Jacobian is completely decomposable, the Abel Jacobi inversion theorem then ensures that this map is surjective: for any integrand $P/(Q\sqrt{S})$ with at most $g$ simple poles, smooth at infinity, and with residues $\pm 1$, there exists constants $(c_i,d_i)$ such that the integral can be written \eqref{eqpi2}.\\

\noindent
\textbf{Example}: We consiser the curve $y^2=4x^5-10x^4-4x^3+9x^2+6x+1$, which has two independent elliptic morphisms
$$\left(\frac{1}{2},\frac{x^2}{2x+1},\frac{x}{\sqrt{2}(2x+1)^2}\right),\quad \left(\frac{3}{4},\frac{(x+1)^2}{4x+2},-\frac{(x+1)}{4(2x+1)^2}\right)$$
The denominator of the $x$ derivative of \eqref{eqpi2} is
$$\left(4d_2d_1\sqrt{2}+2c_1^2c_2+4c_1c_2^2-10c_2c_1+3c_1+c_2\right)x^2+\left(4d_2d_1\sqrt{2}+\right.$$
$$\left. 4c_1^2c_2-6c_1^2+4c_2c_1-4c_2^2+2c_2\right)x+2c_1^2c_2-3c_1^2+2c_2c_1-2c_2^2+c_2$$
where the parameters satisfy the relations
$$d_1^2=c_1(c_1-1)(c_1-\tfrac{1}{2}),\;\; d_2^2=c_2(c_2-1)(c_2-\tfrac{3}{4}).$$
This gives generically two poles, except when the dominant coefficient vanishes for which it has a single pole, with the condition on $c_1,c_2$
$$2c_1^2c_2-4c_1c_2^2+4c_1c_2-3c_1+c_2=0.$$

\subsection{Integration Algorithm}

We can now combine everything. Hermite reduction removes multiple poles. The algorithm \underline{\textit{EllipticDivisors}} allows to represent any divisor with a single pole using third kind elliptic integrals due to the Abel Jacobi map being surjective. The last part without any poles forms a $g$ dimensional vector space of holomorphic forms on the hyperelliptic curve, which can be decomposed using first kind elliptic integrals.\\

\noindent
\underline{\textit{HyperellipticToElliptic}} \\
\textsf{Input:} A hyperelliptic integral $I=\int P/(Q\sqrt{S}) dx$ with completely decomposable Jacobian, and a complete list $L$ of independent elliptic morphisms.\\
\textsf{Output:} An expression of the integral using algebraic logs, and elliptic integrals.\\
\begin{enumerate}
\item Apply \underline{\textit{HermiteReduction}}, and note $\tilde{P}/(\tilde{Q}\sqrt{S})$ the resulting integrand, and $H$ the Hermite part.
\item Compute a partial fraction decomposition
$$\frac{\tilde{P}}{\tilde{Q}}= \sum\limits_{i=1}^\ell \sum\limits_{P_i(\alpha)=0} \frac{T_i(\alpha)}{x-\alpha}$$
\item Note $R,J=\underline{\textit{EllipticDivisors}}(S,L)$
\item For each $i=1\dots \ell$, find a solution in $c_k,d_k$ and $a_j$, with $d_k\neq 0$, of the equality as a function in $\mathbb{C}(x,\sqrt{S(x)})$
$$a_0 J(x)+\sum_{j=1}^{\sharp L} a_j\frac{L_{j,2}'}{L_{j,3}}=\frac{T_i(\alpha)\sqrt{S(\alpha)}}{(x-\alpha)\sqrt{S(x)}} $$
where $\alpha$ is a root of $P_i$, and substitute, giving an expression
$$\tilde{R}_i(x,\sqrt{S(x)},\alpha)=\sum\limits_{j=1}^{\sharp L} a_j(\alpha) \mathcal{F}(F_j(x) \mid \kappa_j)+$$
$$ a_0(\alpha)\left(\ln \left(\frac{R(x,\sqrt{S(x)},\alpha)}{R(x,-\sqrt{S(x)},\alpha)}\right)-\sum\limits_{k=1}^{\sharp L} \Pi(F_k,c_k(\alpha)\mid \kappa_i)\right).$$
\item Return 
$$H(x)+\sum\limits_{i=1}^\ell \sum\limits_{P_i(\alpha)=0} \frac{T_i(\alpha)}{\sqrt{S(\alpha)}} \tilde{R}_i(x,\sqrt{S(x)},\alpha)$$
\end{enumerate}

\begin{proof}[Proof of Theorem \ref{thm3}] 
Step $1$ applies Hermite reduction, and we know that in the completly decomposable case, it will succeed. Step $2$ compute a classical partial decomposition splitting all the poles. Step $3$ computes $\underline{\textit{EllipticDivisors}}$. Step $4$ tries to find parameters $c_i,d_i$ such that the set of poles of $J$ is only one pole at $\alpha$, and adjust the numerators terms using the derivatives of first kind elliptic integrals $\mathcal{F}(F_j(x) \mid \kappa_j)$. We know thanks to Proposition \ref{prop1} that any divisor in the Jacobian can be reached by a sum of third kind elliptic integrals, thus for a suitable choice of the $(c_k,d_k)$ with $d_k\neq 0$ and $a_j$, we can remove the pole in $T_i(\alpha)\sqrt{S(\alpha)}/((x-\alpha)\sqrt{S(x)})$. Now left is an hyperelliptic integral of the first kind. There is a vector space of dimension $g$ of these, and if $L$ is complete, there are $g$ elliptic morphisms. The fact that they are independent ensures that the polynomials $L_{j,2}'/L_{j,3}$ are independent and thus spans this whole $g$ dimensional space. Thus step $4$ always finds a solution, and notes $\tilde{R}_i$ the corresponding integral expression for it. Step $5$ returns the sum of the Hermite part, and the linear combination of the $\tilde{R}_i$.
\end{proof}

\noindent
\textbf{Example}: We obtain the following expression for the hyperelliptic integral
$$\int \frac{dx}{(6x-17)^2\sqrt{4x^5-10x^4-4x^3+9x^2+6x+1}}=$$
$$\frac{-216x^4+648x^3-108x^2-432x-108}{\sqrt{4x^5-10x^4-4x^3+9x^2+6x+1}(49686x-140777)} $$
$$+\frac{207\sqrt{2}}{331240}\mathcal{E}\left(\left.\frac{x^2}{2x+1}\right| \frac{1}{2} \right)-\frac{153}{165620}\mathcal{E}\left(\left.\frac{(x+1)^2}{4x+2}\right| \frac{3}{4} \right)$$
$$-\frac{330243\sqrt{30}}{301428400}\left(\Pi\left(\left.\frac{x^2}{2x+1}, 3\right| \frac{1}{2} \right)   +\Pi\left(\left.\frac{(x+1)^2}{4x+2}, \frac{9}{4}\right| \frac{3}{4} \right)-\right.$$
$$\left.\ln\left(\frac{(4x^2-11x-6)\sqrt{30}+3\sqrt{4x^5-10x^4-4x^3+9x^2+6x+1}}{(4x^2-11x-6)\sqrt{30}-3\sqrt{4x^5-10x^4-4x^3+9x^2+6x+1}} \right) \right)$$
$$-\frac{156581\sqrt{2}}{30142840}\mathcal{F}\left(\left.\frac{x^2}{2x+1}\right| \frac{1}{2} \right)-\frac{454067}{30142840}\mathcal{F}\left(\left.\frac{(x+1)^2}{4x+2}\right| \frac{3}{4} \right)$$

We can see the three parts of the integral. The two first lines are the Hermite reduction, the two next ones is the representation of the divisor using two elliptic integral of the third kind reduced modulo a principal divisor giving the log term. The two last lines is the final expression of the resulting hyperelliptic integral of the first kind, which can be written as a linear combination of elliptic integrals of the first kind.

Remark that the output is not unique. Indeed, in step $4$, the system is linear in the $a_j$ but not in the $c_k,d_k$. And typically the solution is not unique, because the application \underline{\textit{EllipticDivisors}} is surjective but not injective.

\section{Conclusion}
An implementation of these algorithms made on Maple is available at \url{http://combot.perso.math.cnrs.fr/software.html}. The output result is neither optimal, neither complete. It is not optimal in the sense that the number of elliptic integrals is not minimal. For this, we would need to build a $\mathbb{Q}$-basis of the residues, and reduces the corresponding divisors in the Jacobian, and then identify them to elliptic integrals. The positive is that the splitting field of the residues is not needed, the negative is that we cannot use the expression to test, for example, if the expression is elementary. Indeed, elliptic integral of the third kind can be elementary if the second variable is a torsion point, however, compensation could occur between several poles, and this would not be detected.

It is not complete as if the Jacobian is not completely decomposable, a divisor could still possibly be written as a sum of a divisor in the image of \underline{\textit{EllipticDivisors}} and a divisor of torsion. In fact the quotient of the Jacobian by this variety produces an Abelian variety, and testing torsion in it would be necessary to detect such cases.

\bibliographystyle{abbrv}

\begin{small}
\bibliography{super}
\end{small}

\end{document}